\documentclass[11pt]{amsart}
\usepackage{amsfonts}
\usepackage{latexsym,amsfonts,amssymb}
\usepackage{graphicx}
\usepackage{xcolor}
\usepackage[noadjust]{cite}
\usepackage{amsmath}
\usepackage[hidelinks]{hyperref}
\newtheorem{theorem}{Theorem}[section]
\newtheorem{lemma}[theorem]{Lemma}
\newtheorem{corollary}[theorem]{Corollary}

\theoremstyle{definition}

\newtheorem{proposition}[theorem]{Proposition}
\theoremstyle{remark}
\newtheorem{remark}[theorem]{Remark}

\numberwithin{equation}{section}

\allowdisplaybreaks[3]
\usepackage[top=28mm,bottom=25mm,left=25mm,right=25mm]{geometry}
\begin{document}

\title{Biharmonic Lagrangian surfaces with constant Gaussian curvature in $2$-dimensional complex space forms}

\author{Yu Fu}
\address[Y. Fu]{School of Data Science and Artificial Intelligence, Dongbei University of Finance and Economics,
    Dalian 116025, China}
\email{mathyufu@163.com}

\author{Yanru Luo}
\address[Y. Luo]{School of Economics, Dongbei University of Finance and Economics,
Dalian 116025, China}
\email{luoyanru0819@163.com}

\author{Cezar Oniciuc}
\address[C. Oniciuc]{Faculty of Mathematics, Al.I. Cuza University of Iasi, Blvd. Carol I no. 11,
Iasi 700506, Romania}
\email{oniciucc@uaic.ro}




\begin{abstract}
Constant curvature is one of the most fundamental intrinsic conditions in submanifold geometry. Investigating the existence and classification of biharmonic submanifolds under intrinsic curvature constraints is generally difficult, since the biharmonic equation is expressed in terms of extrinsic data. In this paper, we obtain a complete classification of proper, that is non-minimal, biharmonic Lagrangian surfaces with constant Gaussian curvature in 2-dimensional complex space forms. As a consequence, we classify the proper biharmonic Lagrangian surfaces with flat normal bundle and prove that pseudo-umbilical biharmonic Lagrangian surfaces are necessarily minimal.
\end{abstract}

\maketitle \markboth  {Fu et al.}{ Biharmonic Lagrangian surfaces in complex space forms}


\section{Introduction}

The theory of biharmonic maps is a natural and important generalization
of harmonic map theory in differential geometry. Harmonic maps are
critical points of the Dirichlet energy functional, whereas biharmonic
maps are critical points of the bienergy functional introduced by Eells and Sampson~\cite{Eells-Sampson1964, Eells-Sampson1966} (see also Eells
and Lemaire \cite{Eells-Lemaire1983}). For all smooth maps
$\phi\colon (M^m,g)\to\left(\widetilde{M}^n,\widetilde{g}\right)$ between two Riemannian
manifolds, the bienergy functional is defined by
\begin{equation*}
E_{2}\left(\phi\right)=\frac{1}{2}\int_{M}|\tau(\phi)|^{2}\,\mathrm{d}v_{g},
\end{equation*}
where $\tau(\phi)$ denotes the tension field of $\phi$. The Euler-Lagrange equation associated with $E_2$, known as the biharmonic
equation, was derived by Jiang ~\cite{Jiang1986-1,Jiang1986-2}, and it is given by
\begin{equation*}
\tau_{2}(\phi):=\sum_{i=1}^{m}\left\{\left(\nabla^{\phi}_{e_i}\nabla^{\phi}_{e_i}
-\nabla^{\phi}_{\nabla_{e_i}e_i}\right)\tau(\phi)
+\widetilde{R}\left(\tau(\phi),\mathrm{d}\phi(e_i)\right)\mathrm{d}\phi(e_i)\right\}=0,
\end{equation*}
where $\tau_{2}(\phi)$ is the bitension field of $\phi$,
$\{e_1,\ldots,e_m\}$ is a local orthonormal frame field on $M$ and
$\widetilde{R}$ is the curvature tensor field of $\widetilde{M}$.
The biharmonic equation is a fourth-order semilinear elliptic equation,
and every harmonic map is trivially biharmonic. Therefore, the main
interest lies in biharmonic maps that are not harmonic, which are called
proper biharmonic maps.

When $\phi$ is an isometric immersion, that is when $M$ is a submanifold
of $\widetilde{M}$, the biharmonic equation splits into its tangential
and normal components, in which the mean curvature vector field of $M$
and the curvature tensor field of the ambient manifold play a key role.
Since an isometric immersion is harmonic if and only if it is minimal, a
central problem in the theory of biharmonic submanifolds is the
classification of proper, that is non-minimal, biharmonic submanifolds.

The existence and classification of proper biharmonic submanifolds
depend strongly on the curvature of the ambient space and on the
additional geometric conditions imposed on the submanifold. A
long-standing problem in this direction is Chen's conjecture
\cite{Chen1991}, which asserts that every biharmonic submanifold of a
Euclidean space is minimal. Although the conjecture remains open in full
generality, it has been confirmed for several important classes of
submanifolds, including hypersurfaces in Euclidean spaces and submanifolds satisfying suitable curvature conditions;
see, for example, \cite{Fu-Hong-Zhan2021,Fu-Hong-Zhan2023,Fu-Hong-Yang2026, Maeta2026}.

In contrast, proper biharmonic submanifolds do exist in ambient spaces
of positive curvature, most notably in Euclidean spheres, where many
examples and classification results have been obtained. A remarkable
feature is that most of the known examples of proper biharmonic
submanifolds have constant mean curvature (CMC). This observation led to
the Balmu\c{s}--Montaldo--Oniciuc (BMO) conjecture \cite{BMO2008}, which
originally states that every proper biharmonic submanifold of a
Euclidean sphere is CMC, and which can be naturally
formulated in any space form. Chen's conjecture, the BMO conjecture and
other open problems raised in recent decades have played a central role
in the development of proper biharmonic submanifold theory; see, for example,
\cite{Fetcu-Oniciuc2022,Caddeo-Montaldo-Oniciuc2001,Caddeo-Montaldo-Oniciuc2002}.

Interesting results have also been obtained in ambient spaces of
nonconstant curvature. In particular, Lagrangian submanifolds of
complex space forms provide a rich framework in which to study
biharmonicity. Let $\widetilde{M}^{n}(4\epsilon)$ be a complex space
form of constant holomorphic sectional curvature $4\epsilon$. An
$n$-dimensional submanifold $M^n$ of $\widetilde{M}^{n}(4\epsilon)$ is
called Lagrangian if $J(T_pM)=T_p^{\perp}M$ for every $p\in M$, where
$J$ is the complex structure of $\widetilde{M}^{n}(4\epsilon)$. Thus, a
Lagrangian submanifold is a maximal totally real submanifold, and its
geometry reflects a close interaction between the Riemannian and complex
structures of the ambient manifold \cite{Chen2001}. The Lagrangian condition imposes strong
symmetry properties on the second fundamental form of the submanifold but, on the other hand,
it still allows for a rich variety of non-trivial extrinsic geometries.

In the case of Lagrangian surfaces, Sasahara~\cite{Sasahara2007,Sasahara2009} classified
all CMC proper biharmonic and CMC biminimal Lagrangian surfaces, showing in particular that they are rather rigid.
Maeta and Urakawa ~\cite{Maeta2013} classified proper biharmonic Lagrangian surfaces
with parallel normalized mean curvature vector field in complex space
forms. Very recently, Hamiltonian-stationary proper biharmonic Lagrangian
surfaces and proper biharmonic $H$-umbilical Lagrangian
submanifolds in complex space forms were  studied in~\cite{Luo-wang2026}  and \cite{Sasahara2026}, respectively.
Related results on proper biharmonic integral $C$-parallel submanifolds of
$7$-dimensional Sasakian space forms can be found
in~\cite{Fetcu-Oniciuc2012,Sasahara2019}. Further examples and results of proper
biharmonic submanifolds in complex space forms were obtained in~\cite{Fetcu2010, Fetcu-Pinheiro2015, Zhangwei,Bibi2023}.

Despite these advances, the classification of proper biharmonic Lagrangian
surfaces under intrinsic curvature conditions remains largely open.
For surfaces, the Gaussian curvature is the most fundamental intrinsic
invariant, and requiring it to be constant is a natural rigidity
condition. However, the resulting classification problem is difficult,
since the biharmonic equation is expressed in terms of extrinsic
quantities, such as the mean curvature vector field and the second
fundamental form. Since the mean curvature is not assumed to be constant, the biharmonic equation is considerably more difficult to analyze. 
To overcome this difficulty, in our paper we combine the biharmonic
equation with the Gauss and Codazzi
equations of the Lagrangian surface. This yields a system of partial
differential equations involving the mean curvature, the Gaussian
curvature and the connection forms. By a case-by-case analysis
of this system, we prove that any proper biharmonic Lagrangian surface
with constant Gaussian curvature in a complex space form of complex
dimension $2$ has constant mean curvature. Combining this with
Sasahara's classification of  CMC proper biharmonic Lagrangian surfaces ~\cite{Sasahara2007}, we obtain the following
classification result.

\begin{theorem}\label
    {T2}
    Let $M^2$ be a proper biharmonic Lagrangian surface in a
    $2$-dimensional complex space form
    $\widetilde M^2(4\epsilon)$, where
    $\epsilon\in\{-1,0,1\}$. If $M^2$ has constant Gaussian curvature $K$,
    then $K=0$ and $\epsilon=1$.
    Moreover, $M^2$ is an open part of one of the following two Lagrangian tori
    \begin{align}\label{FL}
    \pi\left(\mathbb{S}^1\left(\sqrt{\frac{9\pm\sqrt{41}}{20}}\right)
        \times \mathbb{S}^1\left(\sqrt{\frac{11\mp\sqrt{41}}{40}}\right)
        \times \mathbb{S}^1\left(\sqrt{\frac{11\mp\sqrt{41}}{40}}\right)
        \right),
        \end{align}
        where
        $\pi:\mathbb{S}^5(1)\to\mathbb CP^2(4)$
        is the Hopf fibration.
        \end
        {theorem}
        
We note that proper biharmonic surfaces with constant Gaussian curvature and constant mean curvature in Euclidean spheres were studied in~\cite{Loubeau-Oniciuc2016}; in contrast, here the constancy of the mean curvature is not assumed but obtained, and the curvature of the ambient complex space form introduces terms involving $J$ that have no counterpart in the spherical case.

Theorem~\ref{T2} is quite a rigid result but, interestingly, it parallels the case of minimal, non-totally geodesic Lagrangian surfaces with constant Gaussian curvature in $\widetilde{M}^2(4\epsilon)$. Indeed, it was proved in~\cite{Chen2005, Chen-Ogiue1974, Ejiri1982} that such surfaces exist only for $\epsilon > 0$, in which case they are open parts of a certain flat Lagrangian torus, which is unique up to holomorphic isometries of the ambient manifold.        

As an immediate consequence of Theorem~\ref{T2}, we obtain the following result which provides a
further partial affirmative answer to Chen's conjecture in codimension~$2$
under intrinsic curvature assumptions.

        \begin{corollary}
        Every biharmonic Lagrangian surface of constant Gaussian curvature in
        the complex plane $\mathbb C^2$ is minimal.
        \end{corollary}

Biharmonic Weingarten surfaces with flat normal bundle in $4$-dimensional space
forms were studied in \cite{Andronic-Montaldo-Oniciuc-Sanna2025}. In the present paper, we obtain a complete classification of proper biharmonic
Lagrangian surfaces with flat normal bundle in $2$-dimensional complex space
forms, without any additional assumptions. This follows from Theorem~\ref{T2},
since every Lagrangian surface with flat normal bundle is flat.

       \begin{corollary}
                Let $M^2$
                be a proper biharmonic Lagrangian surface with flat normal
                bundle in a $2$-dimensional complex space form
                $\widetilde M^2(4\epsilon)$, where $\epsilon\in\{-1,0,1\}$. Then $M^2$
                is an open
                part of one of the Lagrangian tori given by
                \eqref{FL}.
                \end{corollary}

                \begin{remark}
                The proper biharmonic Lagrangian tori of \eqref {FL} are not
                pseudo-umbilical. They have constant mean curvature satisfying
                $|H|^
                2=
                \frac{13\pm\sqrt
                    {41}}{8}
                \in\left(0,\frac{5}{2}\right],$
                which agrees with the characterization given in Proposition~2.4
                of~
                \cite
                {Fetcu2010}. Also, those tori have parallel mean curvature vector field $H$ and the shape operator $A_H$ associated to $H$ is parallel as well. Moreover, their corresponding
                $(-4)$-biharmonic integral $C$-parallel surfaces in
                $\mathbb{S}^5(1)\subset\mathbb C^3$
                are locally parametrized by
                $\phi_\pm:\mathbb R^2\to \mathbb{S}^5(1)$, where
\begin{align}
\phi_{\pm}(u,v)=\left(\sqrt{\frac{9\pm\sqrt{41}}{20}}\,\exp\left(-\mathrm{i}\frac{11\mp\sqrt{41}}{2\left(9\pm\sqrt{41}\right)}(u+v)\right),
                \sqrt{\frac{11\mp\sqrt{41}}{40}}\,\exp(
                {\mathrm{i}u}),\right. \label{IntegralCParallelSurface}\\
                \left.\sqrt{\frac{11\mp\sqrt{41}}{40}}\,\exp({\mathrm{i}v})\right). \notag
                \end{align}
                \end{remark}
It is straightforward to verify that if a pseudo-umbilical biharmonic Lagrangian
surface in $\widetilde M^2(4\epsilon)$ is non-minimal, then its Gaussian
curvature would be $K=\epsilon$. On the other hand, by Theorem~\ref{T2}, such a
surface would then have to be one of the two tori described there, which are not
pseudo-umbilical, a contradiction. We therefore obtain the following.

                \begin{corollary}
                    Any pseudo-umbilical biharmonic Lagrangian surface in a
                    $2$-dimensional complex space form
                    $\widetilde M^2(4\epsilon)$
                    is
                    minimal.
                    \end
                    {corollary}

\begin{remark}
In fact, since any pseudo-umbilical biharmonic submanifold $M^m$, with $m\neq 4$, of an
arbitrary Riemannian manifold $\widetilde{M}^n$ is CMC (see
\cite[Remark~4.10]{Loubeau-Montaldo-Oniciuc2008}), the above result follows
directly from Sasahara's result~\cite{Sasahara2007}.
\end{remark}

The paper is organized as follows. In Section~2, we recall some basic facts
about Lagrangian submanifolds in complex space forms and derive the
fundamental equations of proper biharmonic Lagrangian surfaces. We also
establish several preliminary results including, for the sake of completeness, the proof of the
constant mean curvature case, that is Sasahara's result. Section~3 contains the proof of the main
result of the paper. Section~4 is devoted to determining the full period lattices
of the Lagrangian tori~\eqref{FL}.

\section{Some basic results and the PDEs of proper biharmonic Lagrangian surfaces}

Let $\widetilde{M}^{n}(4\epsilon)$ be a complex space form of complex
dimension $n$ with constant holomorphic sectional curvature $4\epsilon$.
The curvature tensor $\widetilde{R}$ of $\widetilde{M}^{n}(4\epsilon)$
is given by
\begin{align*}
	\widetilde{R}(\widetilde{X},\widetilde{Y})\widetilde{Z} =& \widetilde{\nabla}_{\widetilde{X}} \widetilde{\nabla}_{\widetilde{Y}} \widetilde{Z} - \widetilde{\nabla}_{\widetilde{Y}} \widetilde{\nabla}_{\widetilde{X}} \widetilde{Z} - \widetilde{\nabla}_{\left[\widetilde{X}, \widetilde{Y}\right]} \widetilde{Z} \\
	=& \epsilon\Big\{
	\langle \widetilde{Y},\widetilde{Z}\rangle \widetilde{X}
	- \langle \widetilde{X},\widetilde{Z}\rangle \widetilde{Y}
	+ \langle \widetilde{Z},J\widetilde{Y}\rangle J\widetilde{X}
	- \langle \widetilde{Z},J\widetilde{X}\rangle J\widetilde{Y}
	+ 2\langle \widetilde{X},J\widetilde{Y}\rangle J\widetilde{Z} \Big\},
\end{align*}
where $\widetilde{X}$, $\widetilde{Y}$ and $\widetilde{Z}$ are vector fields tangent to the K\"ahler manifold, $\langle\cdot,\cdot\rangle$ denotes the metric and $J$ the
complex structure of $\widetilde{M}^{n}(4\epsilon)$. Every complete,
simply connected complex space form is holomorphically isometric to the
complex Euclidean space $\mathbb{C}^{n}$, the complex projective space
$\mathbb{C}P^{n}(4\epsilon)$, or the complex hyperbolic space
$\mathbb{C}H^{n}(4\epsilon)$, according as $\epsilon = 0$,
$\epsilon > 0$, or $\epsilon < 0$, respectively. After rescaling the
metric, we may assume that $\epsilon \in \{-1, 0, 1\}$.

A connected submanifold of a K\"ahler manifold is called Lagrangian if the
complex structure $J$ maps each tangent space onto the corresponding
normal space.
Let \(M^{n}\) be a Lagrangian submanifold of complex space form \(\widetilde{M}^{n}(4 \epsilon)\). We denote by $\nabla$ and \(\widetilde{\nabla}\) the Levi-Civita connections on \(M^{n}\) and \(\widetilde{M}^{n}(4 \epsilon)\), respectively. We recall the Gauss and Weingarten formulas
\begin{align*}
    &\widetilde{\nabla}_X Y = \nabla_X Y + h(X,Y),\\
    &\widetilde{\nabla}_X \xi = -A_\xi X + D_X \xi,
\end{align*}
where $X$ and $Y$ are vector fields tangent to $M^n$, $\xi$ is a vector field normal to the submanifold, and $h$, $A_\xi$ and $D$ are the second fundamental form, the shape operator and the normal connection, respectively. Since $J$ is parallel, it is easy to check that (cf.~\cite{Chen-Ogiue1974})
\begin{align}
    &D_X JY =  J\left(\nabla_X Y\right), \label{F1}\\
    &A_{JY} X = -Jh(X,Y) = A_{JX} Y. \label{F2}
\end{align}
The Gauss and Codazzi equations are given by
\begin{equation}\nonumber
    \langle R(X,Y) W, Z\rangle = \epsilon \left(  \langle Y,W\rangle \langle X,Z\rangle - \langle Y,Z\rangle \langle X,W \rangle  \right)
    -\left\langle [A_{JW} , A_{JZ}]X,Y \right\rangle, 
    \end{equation}
    \begin{equation}
    \left(\bar{\nabla}_X h\right)(Y,Z) = \left(\bar{\nabla}_Y h\right)(X,Z),\label{F4}
\end{equation}
where $X$, $Y$, $Z$, $W$ are vector fields tangent to $M^n$, and $\bar{\nabla}h$ is defined by
\begin{align}
    \left(\bar{\nabla}_X h\right)(Y,Z) = D_X h(Y,Z) - h\left(\nabla_X Y , Z\right) -h\left(Y,\nabla_X Z\right).\label{F4111}
\end{align}
The mean curvature vector field $H$ is defined by $H=\frac{1}{n} {\rm trace} \, h$, and the function \(|H|\) is called the mean curvature. The submanifold \(M^{n}\) is minimal if $H$ vanishes identically. The Laplacian $\Delta$ acting on smooth functions $f$ on $M^n$ is defined by
\begin{align}
\Delta f=-\sum_{i=1}^n\left(\nabla_{e_i}\nabla_{e_i}f-\nabla_{\nabla_{e_i}e_i}f\right).\nonumber
\end{align}
In the following, we focus on some basic results on  Lagrangian surfaces and derive the system of PDEs of proper biharmonic Lagrangian surfaces in the complex space forms \(\widetilde M^{2}(4 \epsilon)\).

Let $M^2$ be a non-minimal Lagrangian surface in a complex space form
$\widetilde{M}^2(4\epsilon)$. The surface $M^2$ is assumed to be connected. On an open, connected subset $\mathcal{M} \subset M^2$
on which $H$ vanishes nowhere, consider the unit vector field
$e_1 = -JH/|H|$, which is tangent to $M^2$, and let $e_2$ be a unit
vector field tangent to $M^2$ and orthogonal to $e_1$. By using \eqref{F2}, we have
\begin{align}
    h(e_1,e_1) &= (a-c)Je_1 + bJe_2,\nonumber\\
    h(e_1,e_2) &= bJe_1 + cJe_2, \label{F31}\\
    h(e_2,e_2) &= cJe_1 -bJe_2, \nonumber
\end{align}
where $a$, $b$ and $c$ are smooth functions on $\mathcal{M}$ with $a = 2|H| > 0$.

Let \(\omega_{i}^{j}\left(e_{k}\right)=\langle\nabla_{e_{k}} e_{i}, e_{j}\rangle \) be the connection $1$-forms on $M^2$. Using \eqref{F31}, from \eqref{F1} and \eqref{F4111} it follows that
\begin{align*}
&\left(\bar{\nabla}_{e_{1}} h\right)\left(e_{2}, e_{1}\right)=\left(e_{1}(b)+(a-3c) \omega_{1}^{2}\left(e_{1}\right)\right) J e_{1}+\left(e_{1}(c)+3b \omega_{1}^{2}\left(e_{1}\right)\right) J e_{2},\\
    &\left(\bar{\nabla}_{e_{2}} h\right)\left(e_{1}, e_{1}\right)=\left(e_{2}(a-c)-3 b \omega_{1}^{2}\left(e_{2}\right)\right) J e_{1}+\left(e_{2}(b)+(a-3c) \omega_{1}^{2}\left(e_{2}\right)\right) J e_{2},\\
    &\left(\bar{\nabla}_{e_{1}} h\right)\left(e_{2}, e_{2}\right)=\left(e_{1}(c)+3b \omega_{1}^{2}\left(e_{1}\right)\right) J e_{1}-\left(e_{1}(b)-3c \omega_{1}^{2}\left(e_{1}\right)\right) J e_{2},\\
    &\left(\bar{\nabla}_{e_{2}} h\right)\left(e_{1}, e_{2}\right)= \left(e_{2}(b)+(a-3c) \omega_{1}^{2}\left(e_{2}\right)\right) J e_{1}+\left(e_{2} (c)+3b \omega_{1}^{2}\left(e_{2}\right)\right) J e_{2}.
\end{align*}
Since $\omega_{1}^{2} = -\omega_{2}^{1}$, using the Codazzi equation
\eqref{F4}, the above equations reduce to
\begin{align}
    e_{2} (a-c)-e_{1}(b)&=3 b \omega_{1}^{2}\left(e_{2}\right)+(a-3 c) \omega_{1}^{2}\left(e_{1}\right),\label{F32}\\
    e_{1} (b)+e_{2}(c)&=3c \omega_{1}^{2}\left(e_{1}\right)-3b \omega_{1}^{2}\left(e_{2}\right),\label{F33}\\
    e_{2}(b)-e_{1}(c)&=3 b \omega_{1}^{2}\left(e_{1}\right)-(a-3c) \omega_{1}^{2}\left(e_{2}\right).\label{F34}
\end{align}
It is straightforward to see that \eqref{F32} and \eqref{F33} yield
    \begin{align}\label{F35}
    \omega_{1}^{2}\left(e_{1}\right)=  \frac{e_2(a)}{a}.
\end{align}
For the rest of the paper, it is more convenient to work with equations
\eqref{F35}, \eqref{F33} and \eqref{F34} than to try to determine
$\omega^2_1\left(e_2\right)$ explicitly and to derive the relations between the functions
$a$, $b$ and $c$ and their derivatives.

Using the Gauss equation and the definition of curvature tensor $R$ respectively, the Gaussian curvature $K$ can be expressed as
\begin{align}
    K&=a c-2b^{2}-2c^{2}+\epsilon\label{F36}\\
    &=e_2\left(\omega_{1}^{2}\left(e_{1}\right)\right)-e_1\left(\omega_{1}^{2}\left(e_{2}\right)\right)-\left(\omega_{1}^{2}\left(e_{1}\right)\right)^2-\left(\omega_{1}^{2}\left(e_{2}\right)\right)^2.\nonumber
\end{align}
We recall  that  $M^2$ is biharmonic in a complex space form if and only if the following equation holds (cf.~\cite{Fetcu2010})
\begin{align}
    \Delta H = 5\epsilon H,\nonumber
\end{align}
which can be further decomposed into  normal and  tangential components (cf. \cite{Sasahara2007}):
\begin{align*}
    &\Delta^{D} H+\sum_{i=1}^{2} h\left(e_{i}, A_{H} e_{i}\right)=5\epsilon H,\\
    &\sum_{i=1}^{2}\left(A_{D_{e_i} H} e_i + \left(\nabla_{e_i} A_H\right)e_i\right)=0,
\end{align*}
where $\Delta^D$ denotes the rough Laplacian associated with the
normal connection.
With the above notations and \eqref{F35}, we obtain the following characterization of proper biharmonic Lagrangian surfaces $M^2$ immersed in \(\widetilde M^{2}(4 \epsilon)\) (see the details in \cite{Fetcu-Oniciuc2012,Sasahara2007, Luo-wang2026}).
\begin{lemma}[\cite{Luo-wang2026}]
 A non-minimal Lagrangian surface $M^2$ in \(\widetilde M^{2}(4 \epsilon)\) is proper biharmonic if and only if, on $\mathcal{M}$, the following equations hold
    \begin{align}
        &2e_1(a)\omega_{1}^{2}\left(e_{1}\right)+2e_2 (a)\omega_{1}^{2}\left(e_{2}\right)+a\left(e_1\left(\omega_{1}^{2}\left(e_{1}\right)\right)+e_2\left(\omega_{1}^{2}\left(e_{2}\right)\right)-ab\right)=0, \label{F3811}\\
        &-e_1e_1(a)-e_2e_2(a)+\omega_{1}^{2}\left(e_{1}\right)e_2(a)-\omega_{1}^{2}\left(e_{2}\right)e_1(a)+a\left((a-c)^2+2b^2+c^2\label{F3911}\right.\\\left.
        \right.&\left.+\left(\omega_{1}^{2}\left(e_{1}\right)\right)^2+\left(\omega_{1}^{2}\left(e_{2}\right)\right)^2-5\epsilon\right)=0,\nonumber\\
    &\left(\frac{3}{2}a-c\right)e_1(a)+2be_2 (a)+ac\omega_{1}^{2}\left(e_{2}\right)=0,\label{E310}\\
&be_1(a)+\frac1 2(a+4c)e_2(a)-ab\omega_{1}^{2}\left(e_{2}\right)=0.\label{E311}
    \end{align}
\end{lemma}

We consider proper biharmonic surfaces with constant mean curvature in \(\widetilde M^{2}(4 \epsilon)\). This case was settled by Sasahara \cite{Sasahara2007} but we include the proof here for completeness.

\begin{theorem} [\cite{Sasahara2007}]\label{T5}
    Let $M^2$ be a proper biharmonic Lagrangian surface in a 2-dimensional complex space form \(\widetilde{M}^{2}(4 \epsilon)\), where $\epsilon\in\{-1,0,1\}$. If $M^2$ has constant mean curvature, then $M^2$ is flat and $\epsilon = 1$. Precisely, $M^2$ is an open part of the Lagrangian tori \eqref{FL}.
\end{theorem}

\begin{proof}
Since $|H|$ is a non-zero constant, $\mathcal{M} = M^2$ and $a = 2|H|$ is a positive constant on $M^2$.
Then \eqref{F35} gives $\omega_1^2\left(e_1\right) = 0$, and hence
\eqref{F3811}--\eqref{E311} reduce to

        \begin{align}   \label{1FU1}
    \begin{cases}
        &e_2\left(\omega_{1}^{2}\left(e_{2}\right)\right)-ab=0, \\
        &(a-c)^2+2b^2+c^2+\left(\omega_{1}^{2}\left(e_{2}\right)\right)^2-5\epsilon=0,\\
        &c\omega_{1}^{2}\left(e_{2}\right)=0, \\
        &b\omega_{1}^{2}\left(e_{2}\right)=0.
        \end{cases}
        \end{align}
We claim that $\omega_1^2\left(e_2\right) = 0$ on $M^2$. Suppose, to the contrary,
that $\omega_1^2\left(e_2\right) \neq 0$ at any point of an open, connected
subset $\mathcal{U}$ of $M^2$. Then the last two equations of
\eqref{1FU1} yield $b = c = 0$ on $\mathcal{U}$. Combined with \eqref{F34}, this
gives $\omega_1^2\left(e_2\right) = 0$ on $\mathcal{U}$, a contradiction. Therefore
$\omega_1^2\left(e_2\right) = 0$ on $M^2$, and since $\omega_1^2\left(e_1\right) = 0$,
relation \eqref{F36} gives $K = 0$. Moreover, the first equation of
\eqref{1FU1} yields $b = 0$, and the second one reduces to

\begin{align}
(a-c)^2+c^2-5\epsilon=0,\label{C1}
\end{align}
which, in particular, forces $\epsilon$ to be positive, that is $\epsilon = 1$.

Relation \eqref{F36}  implies now
    \begin{align}
ac - 2c^2 + 1 = 0. \label{C2}
        \end{align}
From \eqref{C1} and \eqref{C2} we get
$$
c^2=\frac{7\pm\sqrt{41}}{4} \quad \text{and} \quad a^2=\frac{13\pm\sqrt{41}}{2}.
$$
Consequently, either
$$
a=\sqrt{\frac{13 + \sqrt{41}}{2}}, \quad b=0 \quad \text{and} \quad c= \sqrt{\frac{7 + \sqrt{41}}{4}},
$$
or
$$
a=\sqrt{\frac{13 - \sqrt{41}}{2}}, \quad b=0 \quad \text{and} \quad c= - \sqrt{\frac{7 - \sqrt{41}}{4}}.
$$

\noindent We now recall a construction due to Reckziegel \cite{Reckziegel}, which
yields an explicit representation of the proper biharmonic Lagrangian
surface. Any Lagrangian immersion $\phi\colon M^2 \to \mathbb{C}P^2(4)$
can be lifted locally (and globally, if $M^2$ is simply connected) via
the Hopf fibration $\pi\colon \mathbb{S}^5(1) \to \mathbb{C}P^2(4)$ to a Legendrian
immersion $\bar\phi\colon M^2 \to \mathbb{S}^5(1)$ with $\pi \circ \bar\phi = \phi$.
Moreover, by standard arguments (see \cite{Chen2005, Reckziegel}), the
second fundamental form $\bar h$ of $\bar\phi$ in $\mathbb{S}^5(1)$ is the
horizontal lift of the second fundamental form $h$ of $\phi$, that is
$\mathrm{d} \pi \left(\bar h(X,Y)\right ) = h(X,Y)$. Consequently, $\hat\nabla_X Y$ is given by

    \begin{align}
    \hat{\nabla}_XY=  \nabla_XY+\bar{h}(X,Y)-\langle X, Y\rangle \bar{\phi},\nonumber
 \end{align}
 where $\hat{\nabla}$ is the Levi-Civita connection of $\mathbb R^6\equiv \mathbb C^3$, and $X$ and $Y$ are vector fields tangent to $M^2$.

In our case, since $\omega_1^2 = 0$, the vector fields $e_1$ and
$e_2$ are parallel and commute. Hence there exist local coordinates
$(u, v)$ on $M^2$ such that $\partial_u = e_1$ and $\partial_v = e_2$;
in particular, the metric is $g = \mathrm{d}u^2 + \mathrm{d}v^2$. Using \eqref{F31} with
$b = 0$, we find that the Legendrian lift $\bar\phi$ satisfies the
system of PDEs

        \begin{align}   \label{FC2}
    \begin{cases}
        &\bar{\phi}_{uu}=\mathrm{i}(a-c)\bar{\phi}_{u}-\bar{\phi}, \\
        &\bar{\phi}_{uv}=\mathrm{i}c\bar{\phi}_{v},\\
        &\bar{\phi}_{vv}=\mathrm{i}c\bar{\phi}_{u}-\bar{\phi}.
    \end{cases}
 \end{align}
Integrating \eqref{FC2}, we find that the Legendrian lift $\bar\phi$ is given by
(see also \cite{Sasahara2007})

    \begin{align}
{\bar\phi}(u,v) = \left(
 \sqrt{\frac{c^2}{c^2+1}} \exp\left({-\frac{\mathrm{i}}{c}u}\right),\;
 \sqrt{\frac{1}{c^2+1}} \exp({\mathrm{i}cu}) \cos\sqrt{c^2+1}\,v,\right.\nonumber\\
 \left.
 \sqrt{\frac{1}{c^2+1}} \exp({\mathrm{i}cu}) \sin\sqrt{c^2+1}\,v
 \right).\nonumber
        \end{align}
Finally, acting on ${\bar\phi}$ with the linear holomorphic isometry of $\mathbb{R}^6$
$$
(z_1,z_2,z_3)\to\left(z_1,\frac{z_2+\mathrm{i}z_3}{\sqrt2},\frac{z_2-\mathrm{i}z_3}{\sqrt2}\right),
$$
and making a simple change of parameters of $\mathbb{R}^2$, we obtain the expression~\eqref{IntegralCParallelSurface}. The new parameters are still denoted by $(u,v)$. Moreover, the Hopf cylinder over $\phi_{\pm}$ in $\mathbb{S}^5(1)$, that is $\pi^{-1}\left(\pi\left(\phi_{\pm}\right)\right)$ is an open part of
\begin{align}
    \left(\mathbb{S}^1\left(\sqrt{\frac{9\pm\sqrt{41}}{20}}\right)\times \mathbb{S}^1\left(\sqrt{\frac{11\mp\sqrt{41}}{40}}\right)\times \mathbb{S}^1\left(\sqrt{\frac{11\mp\sqrt{41}}{40}}\right)\right), \nonumber
\end{align}
where $\pi: \mathbb{S}^5(1) \to \mathbb{C}P^2(4)$ is the Hopf fibration. Now the result follows.
\end{proof}

In contrast, classifying proper biharmonic Lagrangian surfaces with constant
Gaussian curvature is considerably more difficult. The main reason is
that the biharmonic equation, which is the Euler--Lagrange equation of
the bienergy functional, is expressed in terms of extrinsic quantities
(the mean curvature vector and the second fundamental form), whereas
the assumption that $K$ is constant is intrinsic. In the next section,
we develop a new technique that yields a complete classification of
proper biharmonic Lagrangian surfaces with constant Gaussian curvature
in complex space forms.

\section{Proof of Theorem \ref{T2}}

Let $M^2$ be a proper biharmonic Lagrangian surface in a complex space form
$\widetilde{M}^2(4\epsilon)$, where $\epsilon \in \{-1, 0, 1\}$ and $M$ is assumed to be connected. In this section, we assume that the
Gaussian curvature $K$ of $M^2$ is constant and 
the main goal is to prove that the mean curvature
$|H|$ is also constant. Assume, to the contrary, that
$\operatorname{grad}|H|^2$ does not vanish identically. Then, there
exists an open, connected subset  $\mathcal{M}$ of $M^2$ on which
$\operatorname{grad}|H|^2 \neq 0$ at any point. On this subset, the
function $|H|^2$ cannot vanish at any point. Consequently, $|H|$
is a smooth, positive function on $\mathcal{M}$ and $\operatorname{grad}|H| \neq 0$ at any point,
that is $a > 0$ and $\left(e_1(a)\right)^2 + \left(e_2(a)\right)^2 > 0$ on $\mathcal{M}$.

Recall from \eqref{F36} that, on $\mathcal{M}$,
\begin{equation}\label{eq:3.1}
  ac - 2b^2 - 2c^2 = K - \epsilon = \text{const},
\end{equation}
where $|H| = a/2$.

In what follows, using the tangential component of the biharmonic equation and Gauss and Codazzi equations, we establish two intermediate results that will be
needed later.

\begin{lemma}\label{le33}
The functions $b$ and $c$ satisfy $b^2c^2>0$ on an open, dense subset of $\mathcal{M}$.
\end{lemma}
\begin{proof}
    We argue by contradiction and assume that
    $b^2c^2=0$ on an open, connected subset $\mathcal{U}$ of $\mathcal{M}$. The proof then splits into the following cases.

    \noindent \textbf{Case A.} $b=c=0$ on $\mathcal{U}$. In this case, it follows from \eqref{E310} and \eqref{E311} that \begin{align*}
        \frac{3}{2}ae_1(a)=0,\quad
        \frac1 2ae_2(a)=0.
    \end{align*} Thus, $e_1(a)=e_2(a)=0$ on $\mathcal{U}$, so $\operatorname{grad}a=0$ there, contradicting $\operatorname{grad}a\ne0$ throughout $\mathcal{M}$.

    \noindent \textbf{Case B.} $b=0$ on $\mathcal{U}$ and $c\ne0$ at any point of an open, connected subset of $\mathcal{U}$,  still denoted by $\mathcal{U}$. Then, we obtain from  \eqref{F33}, \eqref{F34}  and \eqref{F35} that
    \begin{align}
        &e_2(c)=3c\omega_1^2\left(e_1\right)=\frac{3c}{a}e_2(a),\label{1}\\
        &e_1(c)=(a-3c)\omega_1^2\left(e_2\right).\label{3}
    \end{align}
    By \eqref{eq:3.1} we have $ac-2c^2=K-\epsilon$ and differentiating this further yields $$ce_i(a)+(a-4c)e_i(c)=0\quad \text{for}\quad i=1,2,$$ which together with \eqref{1} and \eqref{3} provides
    \begin{align}
        &ce_2(a)+(a-4c)\frac{3c}{a}e_2(a)=\frac{4c}{a}(a-3c)e_2(a)=0,\label{4}\\
        &ce_1(a)+(a-4c)(a-3c)\omega_1^2\left(e_2\right)=0.\label{5}
    \end{align}
 Equations \eqref{4} and \eqref{E311} give $(a-3c)e_2(a)=0$ and $(a+4c)e_2(a)=0$. If $e_2(a) $
 is non-zero, then $a=3c=-4c$, hence $c=0$, a contradiction with assumption $c\neq0$. Therefore  $e_2(a)=0$ on $\mathcal{U}$.
    Now, \eqref{E310} turns into
    \begin{align}
        \left(3a-2c\right)e_1(a)+2ac\omega_{1}^{2}\left(e_{2}\right)=0.\label{2}
    \end{align}
    Since $e_2(a)=0$, it follows that $e_1(a)\neq0$ at any point of $\mathcal{U}$. It is direct to see from \eqref{5} and \eqref{2} that
    \begin{align*}
        3a^3-23a^2c+48ac^2-24c^3=0.
    \end{align*}
Moreover, taking into account $ac-2c^2=K-\epsilon$, the above equation becomes
    \begin{align*}
    4c^6 -8c^4(K - \epsilon) -5c^2(K- \epsilon)^2 +3(K-\epsilon)^3=0,
    \end{align*}
    which is a non-trivial algebraic polynomial equation of $c$ with constant coefficients. Hence, $c$ is a constant and hence $a$ is also a constant. This gives $\operatorname{grad}a=0$ on $\mathcal{U}$, contradicting $\operatorname{grad}a\ne0$ at any point of $\mathcal{U}$.

   \noindent \textbf{Case C.}  $b\ne0$ at some point of $\mathcal{U}$. We can choose an open, connected subset of $\mathcal{U}$, still denoted by $\mathcal{U}$, such that $b\ne0$  at any point of $\mathcal{U}$.  It follows from $b^2c^2=0$ that $c=0$ on $\mathcal{U}$, then we derive from \eqref{eq:3.1} that $2b^2=\epsilon-K$. Hence $e_1(b)=e_2(b)=0$, which together with \eqref{F33}, \eqref{F34} and \eqref{F35}  gives
    $e_2(a)=\omega_1^2\left(e_1\right)=\omega_1^2\left(e_2\right)=0$. By \eqref{E310}, we have $e_1(a)=0$.  Thus, $\operatorname{grad}a=0$ on $\mathcal{U}$, contradicting $\operatorname{grad}a\ne0$ throughout $\mathcal{U}$.
\end{proof}
Now, let us restrict  to an open, connected subset of the dense subset of $\mathcal{M}$ mentioned in Lemma \ref{le33}, still denoted by $\mathcal{M}$, on which $a>0$, $\left(e_1(a)\right)^2 + \left(e_2(a)\right)^2 > 0$ and $b^2c^2 > 0$ at any point. Henceforth, we will work on this open subset $\mathcal{M}$ and, using the tangential component of the biharmonic equation, we establish certain relationships between $e_2(a)$ and $e_1(a)$, $\omega_1^2\left(e_i\right)$, $e_i(b)$, $e_i(c)$, $i=1,2$.
\begin{proposition}
    There exists an open, connected subset of $\mathcal{M}$, still denoted by $\mathcal{M}$, such that $a-4c\ne0$  at any point of $\mathcal{M}$. Moreover, on this subset, the functions $a$, $b$, $c$ and the connection coefficient $\omega_1^2\left(e_2\right)$ satisfy the following system:
    \begin{small}
    \begin{align}\label{E1}
        \begin{cases}
            3abe_1(a)=\left(2(K-\epsilon)-3ac\right)e_2(a),\\
            6a\left(a^2-8(K-\epsilon) \right)e_1(b)=\left(3a^3 - 18a^2c+(K-\epsilon)(40 a -144c)\right)e_2(a),\\
            6a^{2} b \left(8(K-\epsilon) - a^{2}\right)e_2(b)=\left(3 a^{5} - 18 a^{4} c + 18 a^{3} c^{2}- 24 (K-\epsilon)^{2}( a - 4  c) \right.\\\left.\quad\quad\quad\quad\quad\quad\qquad\qquad\qquad\quad
            -a(K-\epsilon)\left(5  a^{2} + 4  ac - 144 c^{2}\right)\right)e_2(a),\\
            3 a b \left( a^{2}- 8 (K-\epsilon) \right)e_1(c)=\left(3 a^{3} c - 9 a^{2} c^{2}- 3 (K-\epsilon) a^{2} + 44 (K-\epsilon) a c \right.\\\left.\quad\quad\quad\quad\quad\quad\,\,\qquad\qquad\qquad- 72 (K-\epsilon) c^{2}- 40 (K-\epsilon)^{2} \right)e_2(a),\\
            3 a^{2} \left(8 (K-\epsilon)- a^{2} \right)e_2(c)=\left(6 a^{4} - 9 a^{3} c- 10 (K-\epsilon) a^{2} - 72 (K-\epsilon) a c- 48 (K-\epsilon)^{2}\right)e_2(a),\\
        6a^2b   \omega_1^2\left(e_2\right)=\left(3a^2 + 6ac+4(K-\epsilon)\right)e_2(a).
        \end{cases}
    \end{align}
\end{small}
\end{proposition}
\begin{proof}
	In the elimination steps below, when we combine two equations and
eliminate an unknown, we multiply and subtract the equations, where no
division is involved in the elimination itself. When we solve an
equation for an unknown and substitute it into another equation, we
first verify that the corresponding coefficient is non-zero.

It follows from Lemma \ref{le33} that the functions $b$ and $c$ do not vanish at any point of $\mathcal{M}$. Combining \eqref{E310} and \eqref{E311}, we eliminate
$\omega^2_1(e_2)$ and obtain
\begin{align*}
3abe_1(a)+\left(ac+4b^2+4c^2\right)e_2(a)=0,
\end{align*}
which together with \eqref{eq:3.1} gives
\begin{align}\label{F111}
3abe_1(a)+\left(3ac-2(K-\epsilon)\right)e_2(a)=0.
\end{align}
Since $a>0$ and $b\ne0$, we solve \eqref{F111} for $e_1(a)$
and obtain
\begin{align}\label{F12}
e_1(a)=\frac{2(K-\epsilon)-3ac}{3ab}e_2(a).
\end{align}
Since $c\ne0$, substituting \eqref{F12} into \eqref{E310} and
solving for $\omega^2_1(e_2)$, we have
\begin{align}\label{F13}
    \omega_1^2\left(e_2\right)=\frac{  3a^2 + 6ac+4(K-\epsilon)}{6a^2b}e_2(a).
\end{align}
Moreover, differentiating \eqref{eq:3.1} with respect to $e_i$ $(i=1,2)$ provides
\begin{align}\label{F14}
    ce_i(a)-4be_i(b)+(a-4c)e_i(c)=0.
\end{align}
By taking $i=1,2$ in \eqref{F14}, we derive from \eqref{F12} that
\begin{align}
    &c\left(2 (K-\epsilon)  - 3 a c\right)e_2(a) - 12 a b^{2} e_1(b)  + 3ab(a- 4 c)e_1(c) =0,\label{FL1}\\
    &c e_2(a) - 4 b e_2(b)  + (a - 4 c)e_2(c) =0.\label{FL2}
\end{align}
\textbf{Claim.} We have $a-4c\ne0$ at any point of an open, dense subset of $\mathcal{M}$.
    
    Otherwise, we suppose that $a-4c=0$ on an open, connected subset $\mathcal{U}\subset\mathcal{M}$, then it follows that $e_2(a)=4e_2(c)$. By \eqref{F35}, \eqref{F12}, \eqref{F13} and \eqref{F14}, we have
    \begin{align*}
        &\omega_1^2\left(e_1\right)=\frac{1}{4c}e_2(a),\quad
            \omega_1^2\left(e_2\right)=\frac{ 18c^2 +(K-\epsilon)}{24bc^2}e_2(a),\\
            &e_1(a)=\frac{(K-\epsilon)-6c^2}{6bc}e_2(a),\quad
        ce_1(a)=4be_1(b).
    \end{align*}
By the above equations, we derive from \eqref{F33} that
        \begin{align*}
            e_{1} (b)+e_{2}(c)&=\frac{(K-\epsilon)-6c^2}{24b^2}e_2(a)+\frac{1}{4}e_2(a)= \frac{3}{4}e_2(a)-\frac{ 18c^2 +(K-\epsilon)}{8c^2}e_2(a),
    \end{align*}
which leads to
\begin{align}\label{E0}
    \left(6c^4-(K-\epsilon)c^2-42b^2c^2-3(K-\epsilon)b^2\right)e_2(a)=0.
\end{align}
If $e_2(a)=0$ on $\mathcal{U}$, then it follows from $e_1(a)=\frac{(K-\epsilon)-6c^2}{6bc}e_2(a)$ that $e_1(a)=0$, so $\operatorname{grad}a=0$ on $\mathcal{U}$, contradicting $\operatorname{grad}a\ne0$ throughout $\mathcal{U}$.
If $e_2(a)\ne0$ at some point of $\mathcal{U}$, then we choose an open subset of $\mathcal{U}$, still denoted by $\mathcal{U},$ such that $e_2(a)\ne0$  at any point of $\mathcal{U}$. It follows from \eqref{E0}, \eqref{eq:3.1} and $a-4c=0$ that
\begin{align*}
72c^4-34(K-\epsilon)c^2-3(K-\epsilon)^2=0,
\end{align*}
which is a non-trivial polynomial equation concerning $c$ with
constant coefficients. It implies that $a=4c$ is constant on $\mathcal{U}$, so $\operatorname{grad}a=0$ there, contradicting $\operatorname{grad}a\ne0$ throughout $\mathcal{U}$.

The Claim shows that $a-4c\ne0$ on an open, dense subset of $\mathcal{M}$. We restrict to an open, connected subset of this subset, still denoted by $\mathcal{M}$. Thus, $a>0$, $\operatorname{grad}a\ne0$, $b^2c^2>0$ and $a-4c\ne0$ throughout $\mathcal{M}$.

Since $a-4c\ne0$ at any point of $\mathcal{M}$, combining \eqref{F34} with \eqref{F35}, \eqref{F13} and \eqref{FL1}, we eliminate the terms  $e_1(c)$ and we obtain
\begin{align}
	&24 a^{2} b^{2} e_1(b) - 6 a^{2} b (a - 4 c)e_2(b)\label{F319}\\
	&=\left(4(K-\epsilon)\left( a^{2} - 6  a c + 12c^{2} \right)+ 3 a^{4} - 15 a^{3} c - 18 a^{2} b^{2} - 12 a^{2} c^{2} + 72 a b^{2} c + 72 a c^{3}\right)e_2(a).\nonumber
\end{align}
Similarly, combining \eqref{F33} with \eqref{F35}, \eqref{F13} and \eqref{FL2}, we eliminate the terms  $e_2(c)$, and we get
\begin{align}
   2a^2(a - 4 c) e_1(b)+8 a^{2} b e_2(b) +  \left(4 (K-\epsilon)(a - 4c) + 3 a^{3} - 14 a^{2} c\right) e_2(a)=0.\label{F320}
\end{align}
Combining \eqref{F319} with \eqref{F320}, we eliminate $e_2(b)$ and obtain
\begin{align}\nonumber
    P_1e_2(a)+Q_1e_1(b)=0,
\end{align}
where
\begin{align*}
    P_1=P_1(a,c)&=3 a^{3} + 18 a^{2} c -72 a b^{2}  -216 a c^{2} +4 (K-\epsilon)a +288 b^{2} c + 288 c^{3}
    \\&=3 a^{3} + (36 a-144c)\left(ac-2b^2-2c^2\right)-18a^2c+4 (K-\epsilon)a
    \\&=3 a^{3} + (K-\epsilon) (36 a-144c)-18a^2c+4 (K-\epsilon)a
    \\&=3a^3 - 18a^2c+40(K-\epsilon) a -144(K-\epsilon) c ,\\
    Q_1=Q_1(a)&=-6 a^{3} +48 a^{2} c- 96 a b^{2} - 96 a c^{2}\\&
    =-6 a^{3} +48 a\left(ac-2b^2-2c^2\right)\\&
    =6a\left(8(K-\epsilon) -a^2\right).
\end{align*}
By \eqref{eq:3.1}, we find
\[
a^2-8(K-\epsilon)
=(a-4c)^2+16b^2>0
\]
on $\mathcal M$.  
It follows that
\begin{align}\label{F322}
    e_1(b)=\frac{3a^3 - 18a^2c+(K-\epsilon)(40 a -144c)}{6a\left(a^2-8(K-\epsilon) \right)}\, e_2(a).
\end{align}
Combining \eqref{F319} with \eqref{F320}, we eliminate $e_1(b)$ and simplifying by means of
\eqref{eq:3.1} we obtain
\begin{align}\label{F323}
    P_2e_2(a)+Q_2e_2(b)=0,
\end{align}
where
\begin{align*}
    &P_2=P_2(a,c)=  3 a^{5} - 18 a^{4} c + 18 a^{3} c^{2}- 24 (K-\epsilon)^{2}( a - 4  c) -a(K-\epsilon)\left(5  a^{2} + 4  ac - 144 c^{2}\right),\\
    &Q_2=Q_2(a,b)=-6a^{2} b \left(8(K-\epsilon) -a^2\right).
\end{align*}
Since $Q_2=6a^2b\left(a^2-8(K-\epsilon)\right)\ne0,
$
we solve \eqref{F323} for $e_2(b)$ and obtain
\begin{align}\label{F324}
    e_2(b)=\frac{3 a^{5} - 18 a^{4} c + 18 a^{3} c^{2}- 24 (K-\epsilon)^{2}( a - 4  c) -a(K-\epsilon)\left(5  a^{2} + 4  ac - 144 c^{2}\right)}{6a^{2} b\left(8(K-\epsilon) -a^2\right)}\, e_2(a),
\end{align}
Substituting \eqref{F322} into \eqref{F33} and \eqref{F324} into \eqref{F34},
respectively, and using \eqref{F35} and \eqref{F13}, we obtain
\begin{align}
    &e_2(c)=\frac{6 a^{4} - 9 a^{3} c- 10 (K-\epsilon) a^{2} - 72 (K-\epsilon) a c - 48 (K-\epsilon)^{2} }{3 a^{2} \left(8 (K-\epsilon)- a^{2} \right)}\, e_2(a),\label{F325}\\
    &e_1(c)=\frac{ 3 a^{3} c - 9 a^{2} c^{2} - 3 (K-\epsilon) a^{2} + 44 (K-\epsilon) a c - 72 (K-\epsilon) c^{2} - 40 (K-\epsilon)^{2}}{3 a b \left( a^{2}- 8 (K-\epsilon) \right)}\, e_2(a).\nonumber
\end{align}
This completes the proof.
\end{proof}

We will further work on the open, connected subset $\mathcal{M}$ chosen above, where $a=2|H|>0$, $\operatorname{grad}a\ne0$, $b^2c^2>0$ and $a-4c\ne0$  at any point and, finally, we will obtain a contradiction. System~\eqref{E1}, which was obtained using the tangential component of the biharmonic equation, holds on this subset.

The outline of the main part of the proof is as follows: using system~\eqref{E1} and ~\eqref{F3811} and ~\eqref{F3911}, that is the normal component of the biharmonic equation, we first derive two relations involving $e_1e_2(a)$, $e_2e_2(a)$ and $e_2(a)$ (see ~\eqref{G330} and ~\eqref{G331}); eliminating $e_1e_2(a)$ yields an equation in $e_2e_2(a)$ and $e_2(a)$ (see~\eqref{G332}). Then, employing system~\eqref{E1} and the Lie brackets, we obtain two additional equations of the same type (see~\eqref{G329} and ~\eqref{G334}). Together, these form a system with three equations in $e_2e_2(a)$ and $e_2(a)$. Eliminating these variables yields a relation in $a$ and $c$ alone (see~\eqref{G337}). Differentiating this relation with respect to $e_2$ produces a second equation in $a$ and $c$ (see~\eqref{G338}). Eliminating $c$ from these two, it gives a polynomial equation in $a$ with constant coefficients on an open, connected subset of $\mathcal{M}$ which forces $a$ to be constant, thereby yielding the desired contradiction.

Now we start the main part of the proof. Substituting \eqref{F35}, together with the first and last equations of system~\eqref{E1}, into \eqref{F3811} and \eqref{F3911}, and then using the second through fifth equations of system~\eqref{E1}, we obtain
\begin{align}
    &{P}_3e_1e_2(a)+{Q}_3e_2e_2(a)+{R}_3\left(e_2(a)\right)^2+{S}_3=0,\label{G330}\\
    &{P}_4e_1e_2(a)+{Q}_4e_2e_2(a)+{R}_4\left(e_2(a)\right)^2+{S}_4=0.\label{G331}
\end{align}

Note that the coefficients \(P_i,Q_i,R_i,S_i\) \((3\leq i\leq 14)\), which appear in \eqref{G330}, \eqref{G331}, and the following equations, are listed separately in the Appendix 1--13. Interestingly, they are non-zero polynomials in $a$, $b$ or $c$, or non-zero constants.

Combining \eqref{G330} and \eqref{G331}, we eliminate $e_1e_2(a)$. After simplifying by the non-zero factor
$6a^2b\left(a^2-8(K-\epsilon)\right)$ and using \eqref{eq:3.1}, one has
\begin{align}\label{G332}
    {P}_5e_2e_2(a) +{Q}_5 \left(e_2(a)\right)^2+ {R}_5=0.
\end{align}
We have thus obtained an equation \eqref{G332} relating $e_2e_2(a)$ and $e_2(a)$, whose coefficients are polynomials depending only on the functions $a$ and $c$. Next, using the Gauss equation and applying the Lie bracket to the functions $a$ and $b$ respectively, we derive two additional equations of the same type, and then we eliminate the terms $e_2e_2(a)$ and $e_2(a)$ from these three equations.

We first show  that $3a^2+6ac+4(K-\epsilon)$ is non-zero on a suitable open subset. If this factor vanished on an open, connected subset, then
$$c=-\frac{3a^2+4(K-\epsilon)}{6a}.$$
By \eqref{F12} and $\operatorname{grad}a\ne0$, we have $e_2(a)\ne0$ at any point. Differentiating the above expression along $e_2$ and comparing with \eqref{F325} yields
$$9a^4+46(K-\epsilon)a^2-16(K-\epsilon)^2=0,$$
which forces $a$ to be constant on an open, connected subset, contradicting $\operatorname{grad}a\ne0$ there. Thus, after possibly restricting $\mathcal{M}$, we may assume that $3a^2+6ac+4(K-\epsilon)\ne0$ throughout this subset. 

We compute the Lie bracket acting on a smooth function $f$:
\begin{align}\label{LB}
    [e_{1} , e_{2}](f) &= e_{1}e_{2}(f) - e_{2}e_{1}(f)\\
    &= \nabla_{e_{1}} e_{2}(f) - \nabla_{e_{2}} e_{1}(f) \nonumber\\
    &=-\omega_{1}^{2}\left(e_1\right) e_{1}(f)- \omega_{1}^{2}\left(e_2\right)e_{2}(f).\nonumber
\end{align}
Taking $f=a$ in \eqref{LB}, and using \eqref{F12},  \eqref{F13} and  \eqref{F35},
we obtain
\begin{align*}
    e_{1}e_{2}(a) - e_{2}\left(-\frac{3ac-2(K-\epsilon)}{3ab}e_2(a)\right)=\frac{3ac-2(K-\epsilon)}{3a^2b}\left(e_2(a)\right)^2-\frac{ 3a^2 + 6ac+4(K-\epsilon)}{6a^2b}\left(e_2(a)\right)^2.
\end{align*}
From the first, third and fifth equations in \eqref{E1}, it follows that
\begin{align}\label{G328}
{P}_6e_1e_2(a) + {Q}_6 e_2e_2(a) +  {R}_6\left(e_2(a)\right)^2=0.
\end{align}
Taking $f=b$ in \eqref{LB}, and then substituting the second, third, and last equations in system \eqref{E1} and \eqref{F35}, we obtain
\begin{equation*}
    \begin{split}
        & e_1 \left( \frac{3a^5 - 18a^4c + 18a^3c^2 - 24(K - \epsilon)^2(a - 4c) - a(K - \epsilon)\left(5a^2 + 4ac - 144c^2\right)}
        {6a^2b\left(8(K - \epsilon) - a^2\right)} \, e_2(a) \right) \\
        &\quad - e_2 \left( \frac{3a^3 - 18a^2c + (K - \epsilon)(40a - 144c)}{6a\left(a^2 - 8(K - \epsilon)\right)} \, e_2(a) \right) \\
        &\quad + \frac{3a^3 - 18a^2c + (K - \epsilon)(40a - 144c)}{6a^2\left(a^2 - 8(K - \epsilon)\right)} \, \left(e_2(a)\right)^2 \\
        &\quad + \frac{3a^5 - 18a^4c + 18a^3c^2 - 24(K - \epsilon)^2(a - 4c) - a(K - \epsilon)\left(5a^2 + 4ac - 144c^2\right)}
        {6a^2b\left(8(K - \epsilon) - a^2\right)} \\
        &\quad \cdot \frac{3a^2 + 6ac + 4(K - \epsilon)}{6a^2b} \,\left(e_2(a)\right)^2 = 0.
    \end{split}
\end{equation*}
Using the first five equations in  \eqref{E1}, we have
\begin{align}\label{G333}
{P}_7e_1e_2(a) +{Q}_7 e_2e_2(a) +  {R}_7\left(e_2(a)\right)^2=0.
\end{align}
 Substituting the last equation of system \eqref{E1} and \eqref{F35} into the Gauss equation \eqref{F36} and using the first two equations and the fourth equation of system \eqref{E1}, we obtain
\begin{align}\label{G327}
{P}_8e_1e_2(a)+{Q}_8e_2e_2(a)+{R}_8\left(e_2(a)\right)^2+S_8=0.
\end{align}
Combining \eqref{G327} with \eqref{G328} and \eqref{G327} with \eqref{G333}, we eliminate $e_1e_2(a)$. Simplifying the resulting equations by means of \eqref{eq:3.1}, we obtain, respectively,
\begin{align}
    &{P}_{9}e_2e_2(a)+{Q}_{9}\left(e_2(a)\right)^2+R_9=0,\label{G329}\\
    &{P}_{10}e_2e_2(a)+{Q}_{10}\left(e_2(a)\right)^2+R_{10}=0.\label{G334}
\end{align}
Combining \eqref{G329} with \eqref{G332} and \eqref{G334} with \eqref{G332}, we eliminate $e_2e_2(a)$. Simplifying by the non-zero common factors, we obtain
\begin{align}
{P}_{11}\left(e_2(a)\right)^2+{Q}_{11}=0,\label{G335}\\
{P}_{12}\left(e_2(a)\right)^2+{Q}_{12}=0.\label{G336}
\end{align}
After possibly restricting $\mathcal{M}$ to a smaller open, connected subset, we may assume that $P_{11}(a)\ne0$ throughout $\mathcal{M}$. Indeed, $P_{11}(a)=3a^2+8(K-\epsilon)$ is a non-zero polynomial in $a$, and $\operatorname{grad}a\ne0$ on $\mathcal{M}$.

Multiplying \eqref{G336} by $P_{11}$ and subtracting
\eqref{G335} multiplied by $P_{12}$, we obtain
\[
P_{11}Q_{12}-P_{12}Q_{11}=0.
\]
Substituting the expressions for $P_{11}$, $P_{12}$, $Q_{11}$ and
$Q_{12}$ into the above relation and factoring, we obtain
        \begin{equation}\label{P13}
            a^2(a - 4c)\left( ac - 2c^2-(K-\epsilon) \right)\left( 3a^2 + 6ac+4(K-\epsilon)\right)\left({P}_{13}c^2+{Q}_{13}c+{R}_{13}\right)=0.
        \end{equation}
By \eqref{eq:3.1}, we have $ac-2c^2-(K-\epsilon)=2b^2>0$. Moreover, $a>0$, $a-4c\neq0$ and $3a^2+6ac+4(K-\epsilon)\neq0$ on the current open 
subset. Hence all factors preceding $ {P}_{13}c^2+{Q}_{13}c+{R}_{13}$ are non-zero.
It follows that
    \begin{align}\label{G337}
        {P}_{13}c^2+{Q}_{13}c+{R}_{13}=0.
    \end{align}
    Now, we consider whether $K-\epsilon=0$ on $\mathcal{M}$ or not.
    If $K-\epsilon=0$ on $\mathcal{M}$, since $c\neq0$, then \eqref{G337} reduces to
    \begin{align}
        -27a^3c(27 a^{8} - 180 \epsilon a^{6})=0,\nonumber
    \end{align}
    which implies that $a$ is constant on the current open, connected subset, contradicting $\operatorname{grad}a\ne0$ there.
    
    If $K-\epsilon\ne0$, then this holds throughout $\mathcal{M}$, since $K$ is constant. Differentiating \eqref{G337} with respect to $e_2$ and using the fifth equation in \eqref{E1}, we obtain
    \begin{align}\label{E338}   \left({P}_{14}c^2+{Q}_{14}c+{R}_{14}\right)e_2(a)=0.
    \end{align}
    Since $a>0$ and $b\ne0$, equation \eqref{F12}, together with
    $\operatorname{grad}a\ne0$, shows that
    $e_2(a)\ne0$
    throughout $\mathcal M$. Hence,  \eqref{E338} gives
    \begin{align}\label{G338}   {P}_{14}c^2+{Q}_{14}c+{R}_{14}=0.
    \end{align}
    Since $K-\epsilon\ne0$, after possibly restricting $\mathcal{M}$ to a smaller open, connected subset, we may assume that $3a^2-8K-10\epsilon$ and $7a^2-24K-12\epsilon$ are non-zero throughout this subset. This follows from the polynomial argument above. Hence, $P_{13}$ and $P_{14}$ are non-zero on this subset. We view the left-hand sides of \eqref{G337} and \eqref{G338} as polynomials in the variable $c$ and eliminate $c$ by computing their resultant. Since both equations hold on this subset, we have
    \[
        \operatorname{Res}_c\left(P_{13}c^2+Q_{13}c+R_{13},P_{14}c^2+Q_{14}c+R_{14}\right)=0
    \]
    there. 
   Taking the resultant with respect to $c$, we obtain
    \begin{align}
    	-93312a^{12}(K-\epsilon)
    	\left(a^2-8(K-\epsilon)\right)
    	\left(3a^2+8(K-\epsilon)\right)^3
    \,\sum_{i=0}^{10}k_{2i}a^{2i}=0.\nonumber
    \end{align}
  Since $a>0$, $K-\epsilon\ne0$, $a^2-8(K-\epsilon)>0$, and
    $P_{11}\ne0$, the prefactor is non-zero. Hence
    \begin{equation}\label{3.37}
    	\sum_{i=0}^{10}k_{2i}a^{2i}=0.
    \end{equation}
   The coefficients in \eqref{3.37} are given in Appendix 13.
    
    The polynomial in \eqref{3.37} is non-trivial. Indeed, if $k_{20}=0$, then $K=16\epsilon/17$. Substituting this relation into $k_{18}$ shows that $k_{18}$  is a non-zero multiple of $\epsilon^2$ and if $\epsilon=0$, then $K-\epsilon=0$, contrary to the present case. Hence at least one coefficient is non-zero. In particular, if $\epsilon=0$, then
    $K-\epsilon\ne0$ gives $K\ne0$, and
$k_{20}=334611K\ne0.$ Thus, no degeneracy occurs when $\epsilon=0$.
    
Since the polynomial in \eqref{3.37} is non-trivial and has constant
coefficients, $a$ is constant on the current open, connected subset
$\mathcal M$. This contradicts
$\operatorname{grad}a\ne0$ throughout $\mathcal M$.

Therefore, the assumption that $\operatorname{grad}|H|^2$ is
non-zero on an open, connected subset is impossible. Hence
$\operatorname{grad}|H|^2=0
$
on $M^2$. Since $M^2$ is  proper,
$|H|$ is a positive constant. Combining these with Theorem \ref{T5} completes the proof of Theorem \ref {T2}.

\section{Period lattices and global embeddings}

We determine the full period lattices of the two Lagrangian tori
in \eqref{FL} and of their inverse images under the Hopf fibration.
Fix one of the two choices of sign (simultaneously) and set
\begin{equation}\label{eq:lattice-radii}
    r^2=\frac{9\pm\sqrt{41}}{20},\quad
    s^2=\frac{11\mp\sqrt{41}}{40},\quad
    \alpha=\frac{s^2}{r^2},
\end{equation}
where $r,s>0$. Thus $r^2+2s^2=1$. Consider the horizontal Legendrian
immersion
\begin{equation}\label{eq:phase-lift}
    \phi_\pm(u,v)
    =\left(r e^{-\mathrm{i}\alpha(u+v)},\;s e^{\mathrm{i}u},\;s e^{\mathrm{i}v}\right)
    \colon\mathbb R^2\longrightarrow \mathbb{S}^5(1).
\end{equation}
On $\mathbb R^2$ and $\mathbb R^3$, respectively, define the flat metrics
\begin{equation}\nonumber
  g_M=s^2\left(\mathrm{d}u^2+\mathrm{d}v^2\right)
    +\frac{s^4}{r^2}\left(\mathrm{d}u+\mathrm{d}v\right)^2,
    \quad g_{M\times \mathbb{R}}=g_M+\mathrm{d}t^2.
\end{equation}
The same notation will be used for the metrics induced on translation
quotients.

\begin{theorem}\label{thm:global-lattice}
    Let
    \begin{equation}\label{eq:hopf-saturation}
        \Phi_\pm(u,v,t)=e^{-\mathrm{i}t}\phi_\pm(u,v)
        =\left(r e^{-\mathrm{i}(\alpha(u+v)+t)},\;
        s e^{\mathrm{i}(u-t)},\;s e^{\mathrm{i}(v-t)}\right).
    \end{equation}
    The full period lattice of $\Phi_\pm$ is
    $\Lambda_\pm=\mathbb Z\lambda_1\oplus\mathbb Z\lambda_2
    \oplus\mathbb Z\lambda_3$, where
    \begin{equation}\label{eq:full-period-lattice}
        \begin{aligned}
            \lambda_1&=2\pi(-r^2,-r^2,-r^2),\\
            \lambda_2&=2\pi(1-s^2,-s^2,-s^2),\\
            \lambda_3&=2\pi(-s^2,1-s^2,-s^2).
        \end{aligned}
    \end{equation}
    The map $\Phi_\pm$ descends to an isometric embedding
    \[
    \overline\Phi_\pm\colon\left(\mathbb R^3/\Lambda_\pm, g_{M\times \mathbb{R}}\right)
    \longrightarrow \mathbb{S}^5(1)
    \]
    whose image is given  by $\mathbb{S}^1(r)\times \mathbb{S}^1(s)\times \mathbb{S}^1(s)$. Moreover, if $\pi\colon \mathbb{S}^5(1)\to\mathbb{C}P^2(4)$ is the Hopf
    fibration, then $f_\pm=\pi\circ\phi_\pm$ has full period lattice
    \begin{equation}\label{eq:projected-period-lattice}
        \Gamma_\pm
        =2\pi\mathbb Z(1-s^2,-s^2)
        \oplus2\pi\mathbb Z(-s^2,1-s^2).
    \end{equation}
    It descends to an isometric Lagrangian embedding
    \[
    \overline f_\pm\colon\left(\mathbb R^2/\Gamma_\pm,g_M\right)
    \longrightarrow\mathbb{C}P^2(4)
    \]
    whose image is given by
    $\pi\left(\mathbb{S}^1(r)\times \mathbb{S}^1(s)\times \mathbb{S}^1(s)\right)$.
\end{theorem}

\begin{proof}
    Fix the sign and suppress the subscripts. The three phases in
    \eqref{eq:hopf-saturation} are given by the linear transformation
    \begin{equation}\nonumber
        A=\begin{pmatrix}
            -\alpha&-\alpha&-1\\
            1&0&-1\\
            0&1&-1
        \end{pmatrix},
        \quad \det A=-(1+2\alpha)=-\frac1{r^2}\ne0.
    \end{equation}
    For $j=1,2$, write $X_j=(u_j,v_j,t_j)\in\mathbb R^3$ and set
    \[
    \delta u=u_2-u_1,\quad
    \delta v=v_2-v_1,\quad
    \delta t=t_2-t_1.
    \]
    Now we have
    \begin{equation}\label{eq:phase-coincidences}
        \Phi(X_1)=\Phi(X_2)
        \quad\Longleftrightarrow\quad
        A(X_2-X_1)\in2\pi\mathbb Z^3.
    \end{equation}
    Consequently, the full period lattice is $2\pi A^{-1}\mathbb Z^3$.
    Writing $A(X_2-X_1)=2\pi(m,n,p)^{\mathsf T}$, we obtain
    \[
    -\alpha(\delta u+\delta v)-\delta t=2\pi m,\quad
    \delta u-\delta t=2\pi n,\quad
    \delta v-\delta t=2\pi p
    \]
    for $m,n,p\in\mathbb Z$. Taking into account that $r^2 + 2s^2 = 1$, we get
    \begin{equation}\nonumber
        \begin{aligned}
            \delta u&=2\pi\left(-r^2m+(1-s^2)n-s^2p\right),\\
            \delta v&=2\pi\left(-r^2m-s^2n+(1-s^2)p\right),\\
            \delta t&=-2\pi\left(r^2m+s^2(n+p)\right).
        \end{aligned}
    \end{equation}
    Thus
    \[
    (\delta u,\delta v,\delta t)
    =m\lambda_1+n\lambda_2+p\lambda_3,
    \]
    which proves \eqref{eq:full-period-lattice}. In particular, $\Lambda$
    is a rank-three lattice, and \eqref{eq:phase-coincidences} shows that the induced map $\overline\Phi_\pm\colon \mathbb R^3/\Lambda_\pm
    \longrightarrow \mathbb{S}^5(1)$ is injective.

    A direct calculation gives
    \begin{align*}
        \Phi^*\langle\cdot,\cdot\rangle_{\mathbb{S}^5}
        &=r^2\left(\alpha(\mathrm{d}u+\mathrm{d}v)
        +\mathrm{d}t\right)^2
        +s^2\left(\mathrm{d}u-\mathrm{d}t\right)^2
        +s^2\left(\mathrm{d}v-\mathrm{d}t\right)^2\\
        &=s^2\left(\mathrm{d}u^2+\mathrm{d}v^2\right)
        +\frac{s^4}{r^2}\left(\mathrm{d}u+\mathrm{d}v\right)^2
        +\mathrm{d}t^2=g_{M\times \mathbb{R}},
    \end{align*}
    where the mixed terms involving $\mathrm{d}t$ vanish because
    $r^2\alpha=s^2$. Since the metric $g_{M\times \mathbb{R}}$ factors to the quotient $\mathbb R^3/\Lambda_\pm $, the immersion $\overline\Phi$ is isometric (we keep the same notation for the induced metric).
    Since its domain is compact and it is injective, it is an embedding.
    The invertibility of $A$ also shows that the three phases can be
    prescribed arbitrarily. Its image is therefore the entire product
    $\mathbb{S}^1(r)\times \mathbb{S}^1(s)\times \mathbb{S}^1(s)$. Notice that the generators in
    \eqref{eq:full-period-lattice} satisfy
    \[
    \left(g_{M\times \mathbb{R}}(\lambda_i,\lambda_j)\right)_{i,j=1}^3
    =4\pi^2\operatorname{diag}\left(r^2,s^2,s^2\right).
    \] 
    We next consider $f=\pi\circ\phi$. In homogeneous coordinates,
    \[
    f(u,v)=\left[r:\;s e^{\mathrm{i}((1+\alpha)u+\alpha v)}:
    s e^{\mathrm{i}(\alpha u+(1+\alpha)v)}\right].
    \]
    As the first coordinate is the non-zero constant $r$, two parameter
    values $z,w\in\mathbb R^2$ have the same image precisely when
    \[
    B(z-w)\in2\pi\mathbb Z^2,
    \quad
    B=\begin{pmatrix}1+\alpha&\alpha\\\alpha&1+\alpha\end{pmatrix}.
    \]
    Since $r^2+2s^2=1$, we obtain
    \[
    \det B=1+2\alpha=\frac1{r^2},\quad
    B^{-1}=\begin{pmatrix}1-s^2&-s^2\\-s^2&1-s^2\end{pmatrix}.
    \]
    Hence the full period lattice of $f$ is $2\pi B^{-1}\mathbb Z^2$,
    which is exactly \eqref{eq:projected-period-lattice}; the induced map
    $\overline f$ is injective.

    With respect to the real Euclidean inner product on $\mathbb C^3$,
    \[
    \langle\phi_u,\mathrm{i}\phi\rangle
    =\langle\phi_v,\mathrm{i}\phi\rangle=-r^2\alpha+s^2=0,
    \quad \langle \mathrm{i}\phi_u,\phi_v\rangle=0.
    \]
    Thus $\phi$ is horizontal and $f^*\omega=0$, where $\omega$ is the
    K\"ahler form of $\mathbb{C}P^2(4)$. Since the Hopf fibration is a
    Riemannian submersion,
    \[
    f^*\langle\cdot,\cdot\rangle_{\mathbb{C}P^2(4)}=\phi^*\langle\cdot,\cdot\rangle_{\mathbb{S}^5}=g_M.
    \]
    It follows that $\overline f$ is an isometric Lagrangian immersion.
    Compactness and injectivity imply that it is an embedding. Finally,
    $\pi\circ\Phi(u,v,t)=f(u,v)$.
    The constant coefficients of $g_M$ show that this torus is flat.
\end{proof}

The phase coordinates in \eqref{eq:phase-lift} are not orthonormal.
The following equivalent formulation gives the period lattice in the standard Euclidean coordinates of $\mathbb R^3$.

\begin{corollary}
    Let $c$ and $\kappa$ be the real numbers defined by
    \[
    c=\pm\sqrt{\frac{7\pm\sqrt{41}}{4}},
    \quad \kappa=\sqrt{1+c^2},
    \]
    where the sign is taken simultaneously.
 Set
    \begin{equation}\nonumber
        \begin{split}
            \widehat\Phi_\pm(x,y,t)
            &=\Phi_\pm(cx+\kappa y,cx-\kappa y,t)\\
            &=\left(r e^{-\mathrm{i}(x/c+t)},\;
            s e^{\mathrm{i}(cx+\kappa y-t)},\;
            s e^{\mathrm{i}(cx-\kappa y-t)}\right).
        \end{split}
    \end{equation}
    Then $\widehat\Phi_\pm$ induces an isometric embedding
    \[
    \left(\mathbb R^3/\widehat\Lambda_\pm,
    \mathrm{d}x^2+\mathrm{d}y^2+\mathrm{d}t^2\right)
    \longrightarrow \mathbb{S}^5(1),
    \]
    where $\widehat\Lambda_\pm
    =\mathbb Z\ell_1\oplus\mathbb Z\ell_2\oplus\mathbb Z\ell_3$ and
    \begin{equation}\label{eq:euclidean-period-lattice}
        \ell_1=\frac{\pi}{1+c^2}(-2c,0,-2c^2),\quad
        \ell_2=\frac{\pi}{1+c^2}(c,\kappa,-1),\quad
        \ell_3=\frac{\pi}{1+c^2}(c,-\kappa,-1).\nonumber
    \end{equation}
    Its image is $\mathbb{S}^1(r)\times \mathbb{S}^1(s)\times \mathbb{S}^1(s)$.
\end{corollary}

\begin{proof}
    The radii in \eqref{eq:lattice-radii} satisfy
    \[
    r^2=\frac{c^2}{1+c^2},\quad
    s^2=\frac1{2(1+c^2)},\quad \alpha=\frac1{2c^2}.
    \]
    The invertible linear change of variables
    \[
    T(x,y,t)=(cx+\kappa y,cx-\kappa y,t)
    \]
    therefore gives
    $T^*g_{M\times \mathbb{R}}=\mathrm{d}x^2+\mathrm{d}y^2+\mathrm{d}t^2$ and
    $\widehat\Phi_\pm=\Phi_\pm\circ T$.
    Applying $T^{-1}$ to the three generators in
    \eqref{eq:full-period-lattice} yields precisely $\ell_1,\ell_2,\ell_3$.
    Thus $\widehat\Lambda_\pm=T^{-1}\Lambda_\pm$, and the conclusion
    follows from Theorem~\ref{thm:global-lattice}. In particular, these
    generators are pairwise orthogonal in the Euclidean metric and satisfy
    \[
    |\ell_1|=2\pi r,\quad |\ell_2|=|\ell_3|=2\pi s.
    \]
\end{proof}

\begin{remark}
    The lift $\phi_\pm$ is singly periodic, whereas its Hopf projection
    $\pi\circ\phi_\pm$ is doubly periodic. Since
    $\alpha=(7\mp\sqrt{41})/4$ is irrational, the periods of $\phi_\pm$
    are precisely $(2\pi m,-2\pi m)$, $m\in\mathbb Z$. Thus $\phi_\pm$
    descends to an injective isometric Legendrian immersion of the
    complete flat cylinder
    \[
    \left(\mathbb R^2/L,g_M\right),
    \quad L=\mathbb Z(2\pi,-2\pi).
    \]
    This cylinder is horizontal with respect to the Hopf
    fibration, since
    \[
    \langle(\phi_\pm)_u,\mathrm{i}\phi_\pm\rangle
    =\langle(\phi_\pm)_v,\mathrm{i}\phi_\pm\rangle=0.
    \]
    In contrast, $\partial_t\Phi_\pm=-\mathrm{i}\Phi_\pm$ is tangent to the Hopf
    fibres. The cylinder immersion is not an embedding. Its image is
    dense in $\mathbb{S}^1(r)\times \mathbb{S}^1(s)\times \mathbb{S}^1(s)$, as follows from the
    irrationality of $\alpha$.
\end{remark}

\section*{Appendix. Explicit formulas for the coefficients arising in \\ the proof of Theorem \ref{T2}}
\label{Appendix}

\subsection*{Appendix 1.} The coefficients in \eqref{G330} are given by
\begin{align*}
	&{P}_3={P}_3(a,b)=36 a^{3} b^{3} \left(a^{2} - 8 (K-\epsilon)\right),\\
	&{Q}_3={Q}_3(a,b,c)=6 a^{2} b^{2} \left(a^{2} - 8 (K-\epsilon)\right)\left(3 a^{2} +6 a c +4 (K-\epsilon)\right),\\
	&{R}_3={R}_3(a,b,c)=9 a^{7} - 36 a^{6} c -36 a^{5} b^{2} - 54 a^{5} c^{2} -3 a^{5} (K-\epsilon) + 108 a^{4} b^{2} c + 108 a^{4} c^{3} \\&\,\,\,\quad\quad\quad\quad- 114 a^{4} c(K-\epsilon)- 144 a^{3} b^{2} (K-\epsilon)+ 480 a^{3} c^{2}(K-\epsilon) - 92 a^{3} (K-\epsilon)^{2} \\&\,\,\,\quad\quad\quad\quad+864 a^{2} b^{2} c (K-\epsilon) +864 a^{2} c^{3} (K-\epsilon) + 128 a^{2} c (K-\epsilon)^{2} +384 a b^{2} (K-\epsilon)^{2} \\&\,\,\,\quad\quad\quad\quad+ 1152 a c^{2} (K-\epsilon)^{2} -96 a (K-\epsilon)^{3} + 384 c (K-\epsilon)^{3},\\
	&{S}_3={S}_3(a,b)=-36 a^{5} b^{4} \left(a^{2} - 8(K-\epsilon)\right).
\end{align*}

\subsection*{Appendix 2.} The coefficients in \eqref{G331} are given by
\begin{align*}
	&   {P}_4={P}_4(a,b,c)=12 a^2 b \left(a^{2} -8 (K-\epsilon)\right) \left(3 a c -2 (K-\epsilon)\right),\\
	&{Q}_4={Q}_4(a,b)=-36 a^{3} b^{2}\left(a^{2} - 8 (K-\epsilon)\right),\\
	&{R}_4={R}_4(a,b,c)=9 a^{6} +72 a^{5} c + 72 a^{4} b^{2} + 72 a^{4} c^{2} -84 a^{4} (K-\epsilon)-192 a^{3} c(K-\epsilon)\\&\,\,\,\quad\quad\quad\quad-576 a^{2} b^{2} (K-\epsilon)-576 a^{2} c^{2} (K-\epsilon)-400 a^{2}(K-\epsilon)^{2} - 768 a c (K-\epsilon)^{2} \\&\,\,\,\quad\quad\quad\quad-128 (K-\epsilon)^{3},\\
	&{S}_4={S}_4(a,b,c)=36a^{4} b^{2} \left(a^{2} - 8(K-\epsilon)\right) \left(a^{2} - 2 a c+ 2 b^{2} +2 c^{2} -5\epsilon\right).
\end{align*}

\subsection*{Appendix 3.} The coefficients in \eqref{G332} are given by
\begin{align*}
	&{P}_5={P}_5(a,c)=6 a^{2} \left(a^{2} - 8(K-\epsilon)\right) \left(a c - 2 c^{2} -(K-\epsilon)\right) \left(18 a^{3} c -15 a^{2}(K-\epsilon) - 8 (K-\epsilon)^{2}\right),\\
	&{Q}_5={Q}_5(a,c)=864a^4c^3\left(a^2 +4(K -\epsilon)\right)-6ac^2\left(99a^6 +144a^4(K-\epsilon) +544 a^2(K-\epsilon)^2 \right.\\\left.\right.&\left.\,\,\,\quad\quad\quad\quad+ 256(K-\epsilon)^3\right)+c\left(3a^2 +4(K-\epsilon)\right)\left(9a^6 + 318a^4(K-\epsilon)  -512a^2(K-\epsilon)^2 \right.\\\left.\right.&\left.\,\,\,\quad\quad\quad\quad- 384(K-\epsilon)^3\right)-a(K-\epsilon)\left(9a^6 + 420a^4(K-\epsilon) + 256a^2(K-\epsilon)^2 - 768(K-\epsilon)^3\right),\\
	&{R}_5={R}_5(a,c)=-18 a^{5} \left(3 a^{2} -5K-10\epsilon\right) \left(a^{2} - 8(K-\epsilon)\right) \left(a c - 2 c^{2} -(K-\epsilon)\right)^{2}.
\end{align*}

\subsection*{Appendix 4.} The coefficients in \eqref{G328} are given by
\begin{align*}
	&{P}_6={P}_6(a,b)=18 a^{3} b^{3}\left(a^{2} - 8(K-\epsilon)\right),\\
	&{Q}_6={Q}_6(a,b,c)=6 a^{2} b^{2}\left(a^{2} - 8 (K-\epsilon)\right)\left(3 a c - 2 (K-\epsilon)\right),\\
	&{R}_6={R}_6(a,b,c)=9 a^{6} c - 27 a^{5} b^{2} - 54 a^{5} c^{2} - 6 a^{5} (K-\epsilon)+ 54 a^{4} b^{2} c + 54 a^{4} c^{3} +21 a^{4} c(K-\epsilon)  \\&\,\,\,\quad\quad\quad\quad+24 a^{3} b^{2}(K-\epsilon)- 48 a^{3} c^{2} (K-\epsilon) + 10 a^{3} (K-\epsilon)^{2} + 432 a^{2} b^{2} c (K-\epsilon)\\&\,\,\,\quad\quad\quad\quad+ 432 a^{2} c^{3} (K-\epsilon)- 64 a^{2} c (K-\epsilon)^{2} + 48 a (K-\epsilon)^{3} - 192 c (K-\epsilon)^{3}.
\end{align*}

\subsection*{Appendix 5.} The coefficients in \eqref{G333} are given by
\begin{align*}
	&{P}_7={P}_7(a,b,c)= 3 a^{2} b\left(a^{2} -8(K-\epsilon)\right) \left(3 a^{5} -18 a^{4} c +18 a^{3} c^{2} - 24(K-\epsilon)^{2} (a-4c) \right.\\\left.\right.&\,\,\,\quad\quad\quad\quad\left.- a(K-\epsilon)\left(5 a^{2} +4 a c -144  c^{2}\right)\right),\\
	&{Q}_7={Q}_7(a,b,c)=3 a^{3} b^{2}\left(a^{2} -8(K-\epsilon)\right)\left( 3 a^{3} - 18 a^{2}c + (K-\epsilon)(40a- 144 c)\right),\\
	&{R}_7={R}_7(a,b,c)=- 27 a^{8} c + 99 a^{7} b^{2} + 54 a^{7} c^{2} - 30 a^{7} (K-\epsilon) - 54 a^{6} b^{2} c - 54 a^{6} c^{3} -87 a^{6} c (K-\epsilon)\\&\,\,\,\quad\quad\quad\quad+ 252 a^{5} b^{2} (K-\epsilon) + 1404 a^{5} c^{2} (K-\epsilon) + 700 a^{5} (K-\epsilon)^{2} - 864 a^{4} b^{2} c(K-\epsilon)\\&\,\,\,\quad\quad\quad\quad- 864 a^{4} c^{3} (K-\epsilon)  - 2056 a^{4} c (K-\epsilon)^{2}- 1344 a^{3} b^{2} (K-\epsilon)^{2} + 8448 a^{3} c^{2} (K-\epsilon)^{2} \\&\,\,\,\quad\quad\quad\quad+ 896 a^{3} (K-\epsilon)^{3} - 17280 a^{2} b^{2} c (K-\epsilon)^{2} - 17280 a^{2} c^{3} (K-\epsilon)^{2} - 11456 a^{2} c (K-\epsilon)^{3}  \\&\,\,\,\quad\quad\quad\quad- 6912 a b^{2} (K-\epsilon)^{3}- 6912 a c^{2} (K-\epsilon)^{3} - 3840 a(K-\epsilon)^{4} + 1536 c(K-\epsilon)^{4}.
\end{align*}

\subsection*{Appendix 6.} The coefficients in \eqref{G327} are given by
\begin{align*}
	&{P}_8={P}_8(a,b,c)=3 a^{2} b\left(a^{2} -8 (K-\epsilon)\right) \left(3 a^{2} + 6 a c +4 (K-\epsilon)\right),\\
	&{Q}_8={Q}_8(a,b)=-18 a^{3} b^{2} \left(a^{2} -8 (K-\epsilon)\right) ,\\
	&{R}_8={R}_8(a,b,c)=2 \left(27 a^{5} c+18 a^{4} b^{2} + 18 a^{4} c^{2} -54 a^{4} (K-\epsilon)+ 144 a^{3} c (K-\epsilon)-144 a^{2} b^{2} (K-\epsilon)\right.\\\left.\right.&\,\,\,\quad\quad\quad\quad\left.- 144 a^{2} c^{2} (K-\epsilon) - 212 a^{2} (K-\epsilon)^{2} +32 (K-\epsilon)^{3}\right),\\
	&S_8={S}_8(a,b)=18 a^{4} b^{2} K \left(a^{2} - 8(K-\epsilon)\right).
\end{align*}

\subsection*{Appendix 7.} The coefficients in \eqref{G329} are given by
\begin{align*}
	&   {P}_{9}={P}_{9}(a,c)= 6 a^{2} \left(a^{2} -8(K-\epsilon)\right) \left(a c - 2 c^{2} -(K-\epsilon)\right)\left(18 a^{3} c - 15 a^{2}(K-\epsilon) -8 (K-\epsilon)^{2}\right),\\
	&   {Q}_{9}={Q}_{9}(a,c)=864 a^{4}c^{3}\left(a^{2}  + 4(K-\epsilon)\right) -6ac^2\left(27a^2 + 8(K-\epsilon)\right)\left(3a^4 + 8a^2(K-\epsilon)\right.\\\left.\right.&\left.\,\,\,\quad\quad\quad\quad+ 32(K-\epsilon)^2\right)-c\left(27a^8 - 1278a^6(K-\epsilon) - 888a^4(K-\epsilon)^2  +3200a^2\right.\\\left.\right.&\left.\,\,\,\quad\quad\quad\quad\cdot(K-\epsilon)^3+ 1536(K-\epsilon)^4\right)+a(K-\epsilon)\left(45a^6 - 708a^4(K-\epsilon)\right.\\&\left.\,\,\,\quad\quad\quad\quad- 1408a^2(K-\epsilon)^2+768(K-\epsilon)^3\right), \\
	&R_9={R}_{9}(a,c)=- 54 a^{5} K\left(a^{2} - 8 (K-\epsilon)\right) \left(a c - 2 c^{2} - (K-\epsilon)\right)^{2}.
\end{align*}

\subsection*{Appendix 8.} The coefficients in \eqref{G334} are given by
\begin{align*}
	&   {P}_{10}={P}_{10}(a,c)= 3 a^{2} \left( 8 (K-\epsilon)-a^{2} \right) \left(a c - 2 c^{2} -(K-\epsilon)\right) \left(27 a^{4} - 144 a^{3} c + 2a(K-\epsilon)\right.\\\left.\right.&\left.\,\,\,\quad\quad\quad\quad\cdot (51a-144c)+ 16 (K-\epsilon)^{2}\right),\\
	&   {Q}_{10}={Q}_{10}(a,c)=3456a^2c^3\left(a^4 + 5a^2(K - \epsilon) + 8(K - \epsilon)^2\right)-6ac^2\left(405a^6 + 2472a^4(K - \epsilon)\right.\\\left.\right.&\left.\,\,\,\quad\quad\quad\quad+ 3136a^2(K - \epsilon)^2- 2560(K - \epsilon)^3 \right)+
	c\left(297 a^{8} + 5058a^6(K - \epsilon)+ 12552 \right.\\\left.\right.&\left.\,\,\,\quad\quad\quad\quad\cdot a^4(K - \epsilon)^2+6016 a^2(K - \epsilon)^3 - 8704(K - \epsilon)^4\right)-a(K - \epsilon)\left(279a^6 \right.\\\left.\right.&\left.\,\,\,\quad\quad\quad\quad+ 3228a^4(K - \epsilon)+ 4768a^2(K - \epsilon)^2- 2560(K - \epsilon)^3\right) ,\\
	&R_{10}={R}_{10}(a,c)=18 a^{2} K \left(a^{2} - 8 (K-\epsilon)\right) \left(a c - 2 c^{2} - (K-\epsilon)\right) \left(3 a^{5} - 18 a^{4} c + 18 a^{3} c^{2} \right.\\\left.\right.&\left.\,\,\,\quad\quad\quad\quad-24 (a-4c) (K-\epsilon)^{2}-  a(K-\epsilon)\left(5 a^{2}+4 a c -144  c^{2} \right)  \right).
\end{align*}

\subsection*{Appendix 9.} The coefficients in \eqref{G335} are given by
\begin{align*}
	&{P}_{11}={P}_{11}(a)= 3 a^{2} + 8(K-\epsilon),\\
	&   {Q}_{11}={Q}_{11}(a,c)= -a^{2}\left(3 a^{2}-8K -10 \epsilon\right) \left(a c - 2 c^{2} -(K-\epsilon)\right).\\
\end{align*}
\subsection*{Appendix 10.} The coefficients in \eqref{G336} are given by
\begin{align*}
		&   {P}_{12}={P}_{12}(a,c)= 124416 a^{7} c^{4} (K-\epsilon)-c^{3}\left(21384 a^{10} - 160704 a^{8} (K-\epsilon)-165888 a^{6} (K-\epsilon)^{2} \right.\\\left.\right.&\left.\,\,\,\quad\quad\quad\quad+ 663552 a^{4} (K-\epsilon)^{3}\right)  + c^{2} \left(9234 a^{11}- 20736 a^{9} (K-\epsilon)- 502848 a^{7} (K-\epsilon)^{2} \right.\\\left.\right.&\left.\,\,\,\quad\quad\quad\quad- 1166976 a^{5} (K-\epsilon)^{3} - 460800 a^{3} (K-\epsilon)^{4} - 172032 a (K-\epsilon)^{5}\right) -c \left( 729 a^{12} \right.\\\left.\right.&\left.\,\,\,\quad\quad\quad\quad+ 11826 a^{10} (K-\epsilon)- 115344 a^{8} (K-\epsilon)^{2} - 568800 a^{6} (K-\epsilon)^{3} - 698112 a^{4} (K-\epsilon)^{4}\right.\\\left.\right.&\left.\,\,\,\quad\quad\quad\quad- 264192 a^{2} (K-\epsilon)^{5}+ 114688(K-\epsilon)^{6}\right) + 243 a^{11} (K-\epsilon)+ 3888 a^{9} (K-\epsilon)^{2} \\&\,\,\,\quad\quad\quad\quad- 51408 a^{7}(K-\epsilon)^{3} - 182592 a^{5} (K-\epsilon)^{4} -73728 a^{3} (K-\epsilon)^{5} + 28672 a(K-\epsilon)^{6} , \\
	& {Q}_{12}={Q}_{12}(a,c)=18 a^{2} \left(a^{2} - 8 (K-\epsilon)\right) \left(a c - 2 c^{2} - (K-\epsilon)\right)\left(72a^4c^3\left(12a^4- a^2(5K + 64\epsilon) \right.\right.\\\left.\left.\right.\right.&\left.\left.\quad\quad\quad\quad-16(K-\epsilon)(7K + 5\epsilon)\right)-2ac^2\left(297a^8 - 27a^6(3K + 64\epsilon) - 12a^4(K-\epsilon)(127K \right.\right.\\\left.\left.\right.\right.&\left.\left.\quad\quad\quad\quad+ 209\epsilon) -16a^2(K-\epsilon)^2(41K+ 10\epsilon)-1152K(K-\epsilon)^3 \right)+c\left(81a^{10} + 9a^8(55K\right.\right.\\\left.\left.\right.\right.&\left.\left.\quad\quad\quad\quad - 112\epsilon)- 6a^6(K-\epsilon)(113K + 562\epsilon) - 152a^4(K-\epsilon)^2(7K + 20\epsilon)+2816a^2K\right.\right.\\\left.\left.\right.\right.&\left.\left.\quad\quad\quad\quad\cdot(K-\epsilon)^3+1536K(K-\epsilon)^4\right)-a(K-\epsilon)\left(81a^8 +9 a^6(9K - 64\epsilon) - 12a^4(K-\epsilon)\right.\right.\\\left.\left.\right.\right.&\left.\left.\quad\quad\quad\quad\cdot(30K+ 89\epsilon)+ 80a^2(K-\epsilon)^2(9K - 2\epsilon) +384K(K-\epsilon)^3 \right)\right).
\end{align*}

\subsection*{Appendix 11.} The coefficients in \eqref{P13} are given by
\begin{align*}
	&{P}_{13}={P}_{13}(a)=5184 a^{6} (K-\epsilon) \left( 3 a^{2} - 8 K -10 \epsilon\right),\\
	&{Q}_{13}={Q}_{13}(a)=- 27 a^{3} \left(27 a^{8} - 90a^{6} \left(5 K - 3 \epsilon\right) +48 a^{4}(K-\epsilon) (68K - 15\epsilon)\right.\\\left.\right.&\left.\quad\quad\quad- 128a^{2} (K-\epsilon) ^2(K + 98\epsilon) - 2048(K-\epsilon) ^3(11K + 10\epsilon)\right),\\
	&{R}_{13}={R}_{13}(a)=(K-\epsilon) \left(1215 a^{10} - 3726a^{8}(3K -\epsilon) + 1080a^{6} (K-\epsilon)(29K+ 22\epsilon)\right.\\\left.\right.&\left.\quad\quad\quad+ 1440a^{4} (14K-45\epsilon) (K-\epsilon) ^2 - 4608a^{2}(K-\epsilon) ^3(79K + 37\epsilon)\right.\\\left.\right.&\left.\quad\quad\quad-2048(26K-35\epsilon)(K-\epsilon) ^4\right).
\end{align*}

\subsection*{Appendix 12.} The coefficients in \eqref{E338} are given by
\begin{align*}
	&   {P}_{14}={P}_{14}(a)= 1728a^6(K -\epsilon)\left(7a^2 - 24K - 12\epsilon\right),\\
	&{Q}_{14}={Q}_{14}(a)= -3 a^{3} \left(189 a^{8} - 36a^{6}(67K- 37\epsilon) + 48a^{4}(K -\epsilon)(461K-250\epsilon)\right.\\\left.\right.&\left.\quad\quad\quad\quad- 1536a^{2} (K -\epsilon)^2(37K+ 15\epsilon) - 1024(K -\epsilon)^3(41K- 68\epsilon)\right),\\
	&{R}_{14}={R}_{14}(a)=81 a^{12} - 270a^{10} \left(3K- \epsilon\right) + 18a^{8} (K -\epsilon)(57K +233\epsilon)\\&\quad\quad\quad\quad+ 120a^{6}(K -\epsilon)^2(369K -382\epsilon) - 320a^{4} (K -\epsilon)^3(701K +185\epsilon)\\& \quad\quad\quad\quad+ 512a^{2}(K -\epsilon)^4(77K +976\epsilon) +4096(K -\epsilon)^5(211K + 157\epsilon).
\end{align*}

\subsection*{Appendix 13.} The coefficients in \eqref{3.37} are given by
\begin{align*}
	k_{20}=&19683 \left(17 K - 16 \epsilon\right),\\
	k_{18}=& - 6561 \left(1811 K^{2} - 2171 K\epsilon + 400 \epsilon^{2}\right),\\
	k_{16}=& 34992 \left(3318 K^{3} - 1508 K^{2} \epsilon - 4291 K\epsilon^{2} + 2506 \epsilon^{3}\right),\\
	k_{14}=&2916 (K-\epsilon) \left( 142009 K^{3} - 1180586 K^{2} \epsilon + 1061637 K\epsilon^{2} - 148960 \epsilon^{3}\right),\\
	k_{12}=& -3888 (K-\epsilon)\left( 5919017 K^{4} - 18458574 K^{3} \epsilon + 16836043 K^{2} \epsilon^{2} \right.\\\left.\right.&\left.-4266198 K \epsilon^{3} - 156288 \epsilon^{4}\right),\\
	k_{10}=& 5184 (K-\epsilon)^{2}\left( 57604244 K^{4} - 116496817 K^{3} \epsilon + 57730646 K^{2} \epsilon^{2} \right.\\\left.\right.&\left.+ 2245911 K \epsilon^{3} - 3066324 \epsilon^{4}\right),\\
	k_ {8} =&-6912(K-\epsilon)^{3}\left( 312832540 K^{4} - 334354007 K^{3} \epsilon - 91616875 K^{2} \epsilon^{2} \right.\\\left.\right.&\left.+ 145888066 K \epsilon^{3} - 24716198 \epsilon^{4}\right),\\
	k_{6}=& 73728 (K-\epsilon)^{4} \left(128999366 K^{4} - 2898332 K^{3} \epsilon - 120439017 K^{2} \epsilon^{2} \right.\\\left.\right.&\left.+ 14370652 K \epsilon^{3} + 14534810 \epsilon^{4}\right),\\
	k_{4}=&-98304 (K-\epsilon)^{5}\left(256980158 K^{4} + 284386295 K^{3} \epsilon - 116487255 K^{2} \epsilon^{2} \right.\\\left.\right.&\left.- 156148160 K \epsilon^{3} + 12326650 \epsilon^{4}\right),\\
	k_{2}=&262144 (K-\epsilon)^{6} \left( 142993244 K^{4} + 326276185 K^{3} \epsilon+ 203016900 K^{2} \epsilon^{2} \right.\\\left.\right.&\left.- 35598275 K \epsilon^{3} - 53542000 \epsilon^{4}\right),\\
	k_{0}=&-1048576 (K -\epsilon )^{7} \left(22958636 K^{4} + 79204723 K^{3} \epsilon + 109786359 K^{2} \epsilon^{2}
	\right.\\\left.\right.&\left.+ 71001250 K \epsilon^{3} + 18083750 \epsilon^{4}\right).
\end{align*}    

\vspace{1cm}

\noindent{\bf Acknowledgement:} {The work was partially supported by NSFC (No.12571051),  Natural Science Foundation of Liaoning Province (No.2025-MS-285).}
\bibliographystyle{amsplain}

\end{document}